\documentclass{amsart}
\usepackage{graphicx}
\usepackage{verbatim}
\newtheorem{thm}{Theorem}[section]
\newtheorem{cor}[thm]{Corollary}
\newtheorem{lem}[thm]{Lemma}
\newtheorem{prop}[thm]{Proposition}
\theoremstyle{definition}
\newtheorem{defn}[thm]{Definition}
\theoremstyle{remark}
\newtheorem{rem}[thm]{Remark}
\numberwithin{equation}{section}

\newcommand{\set}[1]{\left\{#1\right\}}
\newcommand{\Real}{\mathbb R}

\newcommand{\func}[1]{\ensuremath{\mathrm{#1} \:} }

\newcommand{\dist}[0]{\mathrm{dist}}
\newcommand{\re}[0]{\func{Re}}
\newcommand{\im}[0]{\func{Im}}
\newcommand{\inj}[0]{\func{inj}}

\title{Compactness of the space of genus-one helicoids}
\author{Jacob Bernstein and Christine Breiner}
\address{Dept. of Math, Massachusetts Institute of
Technology, Cambridge, MA  02139, USA}
\email{jbern@math.mit.edu}

\address{Dept. of Math,
Johns Hopkins University, Baltimore, MD 21218, USA}
\email{cbreiner@math.jhu.edu}

\thanks{The first author was supported in part by NSF grant DMS 0606629}
\begin{document}
\begin{abstract}
Using the lamination theory developed by Colding and Minicozzi for
sequences of embedded, finite genus minimal surfaces with boundaries
going to infinity \cite{CM5},  we show that the space of genus-one
helicoids is compact  (modulo rigid motions and homotheties).  This
generalizes a result of Hoffman and White \cite{HW}.
\end{abstract}
\maketitle
\section{Introduction}
We investigate the finer geometric
structure of complete,
one-ended, genus-$g$ minimal surfaces embedded in $\Real^3$.  Recall
that Corollary 1.3 in \cite{BB2} determined that every such surface, the space of which we denoted by
$\mathcal{E}(1,g)$, is asymptotic to a helicoid and hence the
terminology ``genus-$g$ helicoid'' is warranted.  We approach this by showing certain compactness properties for
$\mathcal{E}(1,g)$ which will allow us to further bound the geometry
of genus-$g$ helicoids.  For technical reasons, it is much simpler
to treat the case $g=1$, which we do in this paper.  The compactness
properties of $\mathcal{E}(1,g)$ for $g\geq 2$ will be discussed in
\cite{BBCompLam}.

Our main result shows that (after suitably normalizing) the space of
genus-one helicoids is compact.  This generalizes a result of
Hoffman and White \cite{HW} in which they show compactness for the
space of symmetric genus-one helicoids (i.e. genus-one helicoids
that contain two coordinate axes).  Our result establishes a
connection between the asymptotic geometry of a genus-one helicoid
and the geometry of the handle.  Whether such compactness holds in
$\mathcal{E}(1,g)$ when $g\geq 2$ is still unclear.

\begin{thm} \label{MainCpctnessThm}
Let $\Sigma_i\in \mathcal{E}(1,1)$ with each
$\Sigma_i$ asymptotic to $H$, a fixed helicoid. Then, a
sub-sequence converges uniformly in $C^\infty$ on
compact subsets of $\Real^3$ with multiplicity one to $\Sigma_\infty\in
\mathcal{E}(1,1)\cup \set{{H}}$ with $\Sigma_\infty$
asymptotic to (or equaling) $H$.
\end{thm}
\begin{rem}
Translations along the axis show how $H$ may occur as a limit.
\end{rem}
To prove this, we develop a more general compactness result used
also in \cite{BBCompLam}.  We consider a sequence
$\Sigma_i\in\mathcal{E}(1,g,R_i)$, where $R_i\to \infty$, that has
uniform extrinsic control on the position and size of the genus.
Here $\mathcal{E}(e,g,R)$ denotes the set of smooth, connected,
properly embedded minimal surfaces, $\Sigma\subset\Real^3$, so that
$\Sigma$ has genus $g$ and $\partial \Sigma \subset
\partial B_R(0)$ is smooth and has $e$ components. We
show the sequence $\Sigma_i$ has uniformly bounded curvature and so
contain a convergent sub-sequence.
\begin{thm} \label{CpctnessCorg1}
Suppose $\Sigma_i \in \mathcal{E}(1,g,R_i)$ are such that
$1=r_-(\Sigma_i)\geq \alpha r_+(\Sigma_i)$, the genus of each
$\Sigma_i$ is centered at $0$ and $R_i\to \infty$.  Then a
sub-sequence of the $\Sigma_i$ converges uniformly in $C^\infty$ on
compact subsets of $\Real^3$ and with multiplicity one to a surface
$\Sigma_\infty\in \mathcal{E}(1,g)$ and
$1=r_-(\Sigma_\infty)\geq \alpha r_+(\Sigma_\infty)$.
\end{thm}

For the meaning of $r_+(\Sigma), r_-(\Sigma)$, see Definitions
\ref{outerScaleDef} and \ref{innerScaleDef}. Roughly, $r_+(\Sigma)$
measures the extrinsic spread of the handles of $\Sigma$, while
$r_-(\Sigma)$ measures the size of the smallest handle. When $g=1$,
as $r_-(\Sigma)=r_+(\Sigma)$,  rescaling arguments make Theorem
\ref{CpctnessCorg1} particularly strong. For example, it allows one
to give a description of elements of $\mathcal{E}(1,1,R)$ which have
$r(\Sigma)$ small relative to $R$. This is Theorem \ref{effresult},
which says such a surface is close, in a Lipschitz sense, to a piece
of a genus-one helicoid. This is comparable to the description of
the shape of an embedded minimal disks near a point of large
curvature given by Theorem 1.1 in \cite{BB1}.

The proof of Theorem \ref{CpctnessCorg1} uses extensively the
structural theory for embedded minimal surfaces of Colding and
Minicozzi \cite{CM5,CM1,CM2,CM3,CM4}.  In particular, we make use of
three important consequences of their work: the one-sided curvature
estimates -- Theorem 0.2 of \cite{CM4}; the chord-arc bounds for
minimal disks-- Proposition 1.1 of \cite{CY}; and, most importantly,
their lamination theory for finite genus surfaces -- Theorem 0.9 of
\cite{CM5}. As our proof depends most critically on this last
result, we discuss it in some detail in Appendix \ref{CM5Sec} below.
In particular, we summarize Theorem 0.9 of \cite{CM5} as Theorem
\ref{CM509Thm}.

Let us briefly outline the structure of the paper.   In Section
\ref{WeakCompSec}, we  prove uniform curvature bounds for sequences
of minimal surfaces from a larger class than that of Theorem
\ref{CpctnessCorg1}.  These surfaces have uniformly finite topology
normalized so that, intrinsically, the topology does not concentrate
or disappear. As a consequence, we deduce smooth sub-sequential
convergence to a complete, embedded minimal surface in $\Real^3$.
This result is more general than Theorem \ref{CpctnessCorg1}, and
consequently there are more possibilities for the limit surfaces.
While it is implicit in work of Meeks, Perez and Ros
\cite{MPR3,MPR4}, we give our own proof:
\begin{thm} \label{WeakCompThmInt}
Let $\Sigma_i\in \mathcal{E}(e,g,R_i)$ ($e,g\geq 1$) be such that
$0\in \Sigma_i$, $\inj_{\Sigma_i}\geq 1$, $\inj_{\Sigma_i}(0)\leq
\Delta$ and $R_i\to \infty$. Then a sub-sequence of the $\Sigma_i$
converge smoothly on compact subsets of $\Real^3$ with multiplicity
one to a non-simply connected minimal surface in $\cup_{1\leq k\leq
e+g, 0\leq l \leq g} \mathcal{E}(k,l)$.
\end{thm}

We will prove Theorem \ref{WeakCompThmInt} by using the lamination
theory of Colding and Minicozzi.  The key fact is that the weak
chord-arc bounds (i.e. Proposition 1.1 of \cite{CY}) allow us to
show that the sequence $\Sigma_i$ is ULSC, Uniformly Locally Simply
Connected (see Definition \ref{ULSCdef} below). That is, there is a small,
but uniform, \emph{extrinsic} scale on which the sequence is simply
connected (the uniform lower bounds for the injectivity radius
provide such a uniform \emph{intrinsic} scale). This allows the
local application of the work of Colding and Minicozzi for disks.

The uniform upper bound on the injectivity radius at $0$ implies the
existence of a closed geodesic, $\gamma_i$, in each $\Sigma_i$,
close to $0$ and with uniformly bounded length.  Using the
$\gamma_i$ and the lamination theorem, we establish uniform
curvature bounds on compact subsets of $\Real^3$ for the sequence.
Indeed, suppose one had a sequence with curvature blowing-up; then a
sub-sequence would converge to a singular lamination as in \cite{CM5}.  The nature of the convergence implies that any
sequence of closed geodesics in the surfaces, which lie in a fixed
extrinsic ball, must converge (in a Hausdorff sense) to a subset of
the singular axis; in particular this is true of the $\gamma_i$.
However, this must contradict certain chord-arc properties of the surfaces and thus cannot occur. The uniform
curvature bounds, together with Schauder estimates and the
Arzela-Ascoli theorem then prove Theorem \ref{WeakCompThmInt}.

In Section \ref{twosidedsec}, we use Theorem \ref{WeakCompThmInt} to
deduce that, for sequences in $\mathcal{E}(1,g,R_i)$, as long as the
genus stays inside a fixed extrinsic ball and does not shrink off,
then one has convergence to an element of $\mathcal{E}(1,g)$.
Indeed, with such uniform control, the no-mixing theorem of
\cite{CM5} gives a uniform lower bound on the
injectivity radius and so Theorem \ref{WeakCompThmInt} applies; that
the genus remains in a fixed ball implies that the limit surface
must belong to $\mathcal{E}(1,g)$.  The proof of this gives Theorem
\ref{CpctnessCorg1}. In order to show our main application, i.e.
Theorem \ref{MainCpctnessThm}, which we prove in Section
\ref{GlobalSec}, we couple Theorem \ref{CpctnessCorg1} with the fact
that the surfaces are asymptotic to helicoids. The connection
between the convergence on compact subsets of $\Real^3$ and the
asymptotic behavior at the end is made using the holomorphic Weierstrass data.

Throughout, we denote balls in $\Real^3$, centered at $x$ with
radius $r$, by $B_r(x)$, while intrinsic balls in a surface,
$\Sigma$, are denoted by $\mathcal{B}^\Sigma_r(x)$. The norm squared
of the second fundamental form of $\Sigma$ is $|A|^2$.  At various
points we will need to consider $\Sigma\cap B_r(x$) and when we do
so we always assume $\partial B_r(x)$ meets $\Sigma$ transversely.

\section{Curvature bounds at the Scale of the Topology} \label{WeakCompSec}
As mentioned in the introduction, the lamination theory of
\cite{CM5} proves crucial for the proof of Theorem
\ref{WeakCompThmInt}. Indeed, this implies that there is a
sub-sequence with either uniform curvature bounds or one of two
possible singular convergence models. In the latter case, a simple
topological argument will rule out one singular model and so imply
that the sequence behaves like the blow-down of a helicoid -- i.e.
like Theorem 0.1 of \cite{CM4}. This contradicts the upper bound on
the injectivity radius at the origin, which proves the desired
curvature bounds.

\subsection{Technical lemmas}
In order to obtain the curvature bounds, we will need four technical lemmas.
We first note the following simple topological consequence of the maximum principle:
\begin{prop} \label{BndryCompBndProp}
Let $\Sigma\in \mathcal{E}(e,g,R)$ and suppose $B_r(x)\subset
B_R$. Then, for any component $\Sigma_0$ of $\Sigma\cap
B_{r}(x)$, $\partial \Sigma_0$ has at most $g+e$ components.
\end{prop}
\begin{proof}
Let $\Sigma_i$, $1\leq i \leq n$, be the components of
$\Sigma\backslash \Sigma_0$. The $\Sigma_i$ are smooth compact
surfaces with boundary, and so $\partial \Sigma_i$ is a finite
collection of circles. Thus, the Euler characteristic satisfies
$\chi(\Sigma)=\sum_{i=0}^n \chi(\Sigma_i)$.  By the classification
of surfaces, $\chi(\Sigma)=2-2g-e$ and $\chi(\Sigma_i)=2-2g_i-e_i$
where $g_i$ is the genus of $\Sigma_i$ and $ e_i$ number of
components of $\partial \Sigma_i$.   Note that,  $\sum_{i=0}^n g_i
\leq g$ and $\sum_{i=1}^n e_i=e_0+e$.  Thus, we compute that
$e_0=n+g-\sum_{ i = 0}^n g_i$. The maximum principle implies that
$n\leq e$ (as any $\Sigma_i$, for $i\geq 1$, must meet $\partial
B_R$).  Thus, $e_0\leq e+g$.
\end{proof}

Next we note it is impossible to minimally immerse a (intrinsically) long and thin cylinder in $\Real^3$:
\begin{lem} \label{NoThinTubesLem}
Let $\Gamma$ be a minimal surface with genus $g$ and with $\partial
\Gamma=\gamma_1\cup \gamma_2$ where the $\gamma_i$ are smooth and
satisfy $\int_{\gamma_i} 1+|k_g| \leq C_1$.
Then, there exists $C_2=C_2(g,C_1)$ so that $\dist_{\Gamma}
(\gamma_1,\gamma_2)\leq C_2$.
\end{lem}
\begin{proof}
It is clear that there are points $p_1, p_2\in \Real^3$ so that
$\gamma_i\subset B_{C_1} (p_i)$.  Thus, up to rotating and
translating $\Gamma$, $\gamma_i \subset \set{x_1^2+x_2^2\leq C_1^2}$
and so by the convex hull property, $\Gamma$ itself lies in this
cylinder. Notice that $|\nabla \dist_{\mathbb{S}^2}(\mathbf{n},
\mathbf{n}_0)| \leq |A|$ (here $\mathbf{n}$ is the normal to
$\Gamma$) and so there exists a uniform constant $\delta_0$ so that
if $s \sup_{\mathcal{B}_s(p)} |A|\leq \delta_0$ then
$\mathcal{B}_s(p)$ can be expressed as a graph over $T_p\Gamma$
where the graph has gradient $\leq 1/100$. Indeed, for such a $p$ if
$q\in \mathcal{B}_s(p)$ then $|q-p|\geq \frac{9}{10} \dist_\Gamma
(p,q)$. We refer the reader to Section 2 of \cite{CM2} for more
details. Thus, if $p\in \Gamma$ with $\mathcal{B}_{4C_1}(p) \cap
\partial \Gamma=\emptyset$ and $4C_1 \sup_{\mathcal{B}_{4C_1}(p)}
|A|\leq \delta_0$ then $\mathcal{B}_{4C_1}(p)$ cannot lie in the
cylinder $ \set{x_1^2+x_2^2\leq C_1^2}$. We claim that if
$\dist_\Gamma(\gamma_1,\gamma_2)$ was very large we could find such
a $p$ yielding the desired contradiction.

We use the intrinsic version of a result of Choi-Schoen
\cite{ChoiSchoen} on $\Gamma$.  That is, there exists a universal
$\epsilon>0$ so that if $\mathcal{B}_{s}(p)\subset \Gamma$ is
disjoint from $\partial \Gamma$ and $\int_{\mathcal{B}_s} |A|^2 \leq
\delta \epsilon$ then $\sup_{\mathcal{B}_{s-t} (p)} t^2 |A|^2 \leq
\delta$. We emphasize that $s$ is not a priori restricted to be
smaller than $\inj_\Gamma (p)$.  By the Gauss-Bonnet theorem,
$\int_{\Gamma} |A|^2 \leq 8\pi g+4 C_1=C_1'$.
 If $\dist_\Gamma(\gamma_1,\gamma_2)\geq \frac{32 C_1' C_1}{\delta_0^2 \epsilon}$ then by the pigeonhole principle, there is a $p\in \Gamma$ so that $\mathcal{B}_{8C_1}(p)\cap \partial \Gamma=\emptyset$
 and $\int_{\mathcal{B}_{8C_1}(p)} |A|^2 \leq \delta_0^2 \epsilon$.  The Choi-Schoen theorem and the conclusion of the preceding paragraph then imply that  $C_2\leq \frac{32 C_1' C_1}{\delta_0^2 \epsilon}$.
\end{proof}

We need that minimizing geodesics in almost flat surfaces are almost straight:
\begin{lem} \label{MinGeodesicLem}
Let $u:D_{3/2}(0)\to \Real$ and let $\Sigma$ be the graph of  $u$.  Then for any $\delta>0$ there
is an $\epsilon>0$  so: if $||u||_{C^2}\leq \epsilon$ and $\gamma\subset \Sigma$ a minimizing geodesic with $\partial \gamma=\set{p_-,p_+}\in
\partial B_{1}(0)$ satisfies $\gamma\cap B_{\epsilon}(0)\neq \emptyset$, then there is a line
$0\in L$ so that the Hausdorff distance between $\gamma$ and $L \cap
D_1(0)$ is less than $\delta$.
\end{lem}
\begin{proof}
Fix $\delta>0$ and suppose this result was false. Then there is a sequence of $\Sigma_i$ and $u_i$ with $u_i\to 0$ in $C^2$ and
points $p^i_\pm$ connected by minimizing geodesic $\gamma_i\subset
\Sigma_i$ meeting $B_{\epsilon_i}(0)$ where $1>\epsilon_i\to 0$,
but the conclusion of the lemma does not hold.

Chord-arc bounds for graphs with small $C^1$ bounds imply
$\dist_{\Real^3} (p_+^i,p_-^i)\geq 1/2$. Letting $u_i \to 0$, we get
$p_\pm ^i\to p_\pm^\infty\in \partial D_1(0)$ and
$|p_-^\infty-p_+^\infty|\geq 1/2$.
Now, let $L$ be the line connecting $p_\pm^\infty$.  We claim that
$0\in L$.  If not, then $\ell(L\cap D_1(0))=2-4\alpha$ for some
$\alpha>0$.  Let $L_i$ be the graph (of $u_i$) over $L\cap D_1(0)$,
so $L_i$ is a segment that is a subset of $\Sigma_i$. Clearly, for
some $i_0$,  $i\geq i_0$ implies $\dist_{\Sigma_i}(p^i_\pm,
L_i)<\alpha$ and $\ell(L_i)<\ell(L\cap D_1(0))+\alpha$; and thus, $
\ell(\gamma_i)<2-\alpha$. On the other hand, the hypotheses give a
$p^i\in \gamma_i$ with $p^i\to 0$. Thus, by possibly increasing
$i_0$, for $i\geq i_0$, $\ell(\gamma_i)\geq 2-\alpha$. Similar
arguments show, for $i$ large, $L\cap D_1(0)$ must be Hausdorff
close to $\gamma_i$, yielding the desired contradiction.
\end{proof}

Our final technical lemma shows that, for a ULSC sequence $\Sigma_i$
converging to a minimal lamination with a single singular line, any
closed geodesics in the $\Sigma_i$ that lie in a fixed extrinsic
ball must collapse to the singular set.
\begin{lem} \label{GeoInThinTubeLem} Fix $\Delta>0$.
Suppose $\Sigma_i \in \mathcal{E}(e,g,R_i)$, $R_i \to \infty$,
the $\Sigma_i$ converge to the singular lamination $\mathcal{L}$
with singular set $\mathcal{S}=\mathcal{S}_{ulsc}$ in the sense of
Theorem 0.9 of \cite{CM5}, and that $\mathcal{S}$ is the
$x_3$-axis. Then, given $\epsilon>0$ there exists an $i_0$ so that
for $i\geq i_0$, if $\gamma_i$ is an embedded closed
 geodesic in $\Sigma_i$,
with $\gamma_i\subset
B_{\Delta}$, then $\gamma_i \subset {T}_\epsilon (\mathcal{S})$,
the extrinsic $\epsilon$-tubular neighborhood of $\mathcal{S}$.
\end{lem}
\begin{proof}
 Suppose the lemma was not true; then there exists a sub-sequence of the $\Sigma_i$ so that $\gamma_i$
intersects $K_\epsilon=\overline{B}_{\Delta}(0)\backslash
{T}_{\epsilon} (\mathcal{S})$. Choose $p_i\in \gamma_i \cap
\overline{K}_\epsilon$ so that $x_1^2+x_2^2=\rho^2$ achieves its
maximum on $\gamma_i$ at $p_i$. After possible passing to a further
sub-sequence, $p_i \to p_\infty\in \overline{K}_\epsilon$ with
$\rho^2(p_\infty)\geq \epsilon$. Translate so that
$x_3(p_\infty)=0$. The convergence of \cite{CM5} implies that for
sufficiently large $i$, the component $\Gamma_i$ of $B_{\epsilon/2}
(p_\infty)\cap \Sigma_i$ containing $p_i$ converges smoothly, as a
graph, to the disk $B_{\epsilon/2} (p_\infty)\cap \set{x_3=0}$.  Let
$\gamma'_i$ be the component of $B_{\epsilon/2}(p_\infty)\cap
\gamma_i$ containing $p_i$ with boundary points $q_i^\pm$.

For a given $\delta>0$, there exists $i$ large so that $\Gamma_i$,
$\gamma_i'$ satisfy the hypotheses of Lemma \ref{MinGeodesicLem}
(after a rescaling).  Notice that for large $i$, $\Gamma_i$ is very
flat and in particular is geodesically convex; hence $\gamma_i'$ is
the minimizing geodesic connecting $q_i^\pm$. If $L_i$ is the line
given by the lemma, $\gamma_i'$ lies in the $\delta$-tubular
neighborhood of $L_i\cap D_{\epsilon/4}(p_\infty)$. By passing to a
sub-sequence, the $\gamma_i'$ converge to a segment of a line $L$ in
$\set{x_3=0}$ that goes through $p_\infty$.  Because $\set{x_3=0}$
is transverse to the axis of the cylinder $\set{y\in \Real^3|
\rho^2(y) \leq \rho^2(p_\infty)}$ and $p_\infty$ is
on the boundary of this cylinder, $L$ cannot lie entirely within it.
However, this implies there are points $q_i\in \gamma_i$ with
$\lim_{i\to \infty} \rho^2(q_i) > \rho^2(p_\infty)$, giving a
contradiction.
\end{proof}

\subsection{Proof of Theorem
\ref{WeakCompThmInt}}\label{wkcompactthmsec} We apply the preceding
lemmas in order to show uniform curvature bounds.  We note that
similar techniques are used by Meeks and Rosenberg in \cite{MRDuke}
to achieve different results:
\begin{lem} \label{boundedcurvlem}
 Let $\Sigma_i\in \mathcal{E}(e,g,R_i)$ be such that $0\in \Sigma_i$, $\inj_{\Sigma_i}\geq 1$, $\inj_{\Sigma_i}(0)\leq \Delta$ and $R_i\to \infty$. Then a sub-sequence of the $\Sigma_i$
 satisfy, for any compact $K$,
\begin{equation}
\sup_{i} \sup_{K\cap \Sigma_i} |A|^2 <\infty.
\end{equation}
\end{lem}
\begin{proof}
If this was not the case then by the lamination theorem of
\cite{CM5} a sub-sequence of the $\Sigma_i$ would converge to a
singular lamination $\mathcal{L}$.  For any $x\in \Sigma_i$ with
$|x|\leq R_i/2$, the injectivity radius lower bound and the weak
chord-arc bounds of \cite{CY} imply that there is a $\delta_0>0$ so
the component of $B_{\delta_0}(x) \cap \Sigma_i$ containing $x$ is a
subset of $\mathcal{B}^{\Sigma_i}_{1/2}(x)$. For any fixed $x\in
\Real^3$, as long as $i$ is large enough so $|x|\leq R_i/2$, this
implies that every component of $B_{\delta_0}(x)\cap \Sigma_i$ is a
disk. Thus, the sequence of $\Sigma_i$ is ULSC and so
$\mathcal{S}=\mathcal{S}_{ulsc}$. Hence, after rotating if needed,
$\mathcal{L}=\set{x_3=t}_{t\in \Real}$ and $\mathcal{S}$ is parallel
to the $x_3$-axis and consists of either one or two lines. By Remark
\ref{twosingremark} and Lemma \ref{BndryCompBndProp} we see that
$\mathcal{S}$ must consist of exactly one singular line.
\begin{figure}

\begin{center}
 \includegraphics[width=1.5in,bb=0 0 221 300]{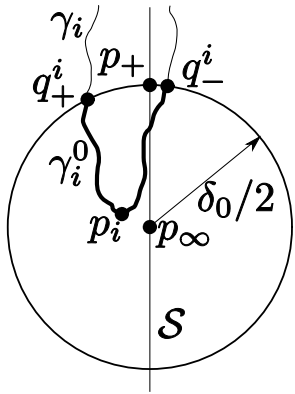}
\caption{The points of interest in the proof of Lemma
\ref{boundedcurvlem}.}\label{LowPointGammaFig}
\end{center}
\end{figure}
We next rule out any singular behavior. The injectivity bound at $0$
and the non-positive curvature of $\Sigma_i$ imply the existence of
$0\in\gamma'_i\subset \Sigma_i$, a geodesic lasso (i.e. a closed
curve that is a geodesic but for at most one point) with
$\ell(\gamma'_i)\leq 2 \Delta$.  The length bound implies
$\gamma'_i\subset B_{3\Delta}$. The fact that $\Sigma_i$ has
non-positive curvature implies that $\gamma'_i$ is not
null-homotopic.  We may minimize in the homotopy class of
$\gamma'_i$ to obtain a closed geodesic $\gamma_i$; note that Lemma
\ref{NoThinTubesLem} allows us to do this even though $\Sigma_i$ has
boundary that is not a priori convex.  Indeed, either $\gamma_i'$
intersects $\gamma_i$ and so $\gamma_i\subset B_{6\Delta}$ or, as
$\ell(\gamma_i)\leq \ell(\gamma_i')$, Lemma \ref{NoThinTubesLem}
implies  that $\dist_{\Sigma_i}(\gamma_i', \gamma_i)\leq C$ for some
large (but uniform) $C$, and so $\gamma_i\subset
B_{(C+6)\Delta}(0)$.

Now for $\epsilon>0$, let ${T}_\epsilon(\mathcal{S})$ be the
extrinsic $\epsilon$-tubular neighborhood of $\mathcal{S}$. Lemma
\ref{GeoInThinTubeLem} then implies there is an $i_\epsilon$ so that
for all $i\geq i_\epsilon$, $\gamma_i\subset
{T}_\epsilon(\mathcal{S})$.

For each $i$, fix $p_i\in \gamma_i$ so that
$x_3(p_i)=\min\set{x_3(p):p\in \gamma_i}$, i.e. the lowest point of
$\gamma_i$. Then a sub-sequence of the $p_i$ converge to $p_\infty
\in \mathcal{S}$.  Let $p_+$ be the point of intersection $\partial
B_{\delta_0/2}(p_\infty)\cap \mathcal{S}$, chosen so
$x_3(p_+)>x_3(p_\infty)$ with $\delta_0$ as before. Pick $i_0$ large
enough so for $i\geq i_0$, $|p_\infty-p_i|\leq \delta_0/4$. The
choice of $\delta_0$ implies that $\gamma_i$ is not contained in
$B_{\delta_0/2}(p_\infty)$; but for $i\geq i_0$, $\gamma_i$ meets
this ball.  Let $\gamma_i^0$ be an arc of $\gamma_i$ in
$B_{\delta_0/2}(p_\infty)$ through $p_i$. Denote $\partial
\gamma_i^0=\{q_i^+,q_i^-\} \subset\partial B_{\delta_0/2}(p_\infty)$
(see Figure \ref{LowPointGammaFig}). For $i\geq i_0$,
$\dist_{\Sigma_i}(q_i^-,q_i^+)\geq\delta_0/3$.  Indeed, for $i\geq
i_0$, the length of $\gamma_i^0$ is bounded below by $\delta_0/2$,
as $|q_i^\pm-p_i|\geq\delta_0/4$.  On the other hand, the lower
bound on the injectivity radius implies either
$\dist_\Sigma(q_i^-,q_i^+)>1/2$ or both lie in a geodesically convex
region with $\gamma_i^0$ the minimizing geodesic connecting them.

By Lemma \ref{GeoInThinTubeLem}, for any $\delta>0$, there is an
$i_\delta \geq i_0$ so for $i\geq i_\delta$, $q_i^\pm \in B_{\delta}
(p_+)$. By the one-sided curvature estimate of \cite{CM4}, there
exist $c>1$ and $1>\epsilon>0$ so that if $\Sigma_1, \Sigma_2$ are
disjoint embedded disks in $B_{cR}$ with $\partial \Sigma_j \subset
\partial B_{cR}$ and $B_{\epsilon R}\cap \Sigma_j \neq \emptyset$,
then for all components $\Sigma_1'$ of $B_R\cap \Sigma_1$ that
intersect $B_{\epsilon R}$, $\sup_{\Sigma_1'} |A|^2 \leq R^{-2}$.
Thus, for $\delta$ such that $0<\frac{2c}{\epsilon}
\delta<\delta_0$, because $\lim_{i \to \infty} \sup_{\Sigma_i \cap
B_{\delta}(p_+)} |A|^2 \to \infty$, there is an $i_\delta'\geq
i_\delta$ so for $i\geq i_\delta'$ there is only one component of
$B_{c\delta/\epsilon}(p_+)\cap \Sigma_i$ that meets
$B_{\delta}(p_+)$.
And thus, there is a $\sigma_i \subset B_{c\delta/\epsilon}(p_+)\cap
\Sigma_i$ that connects $q_i^\pm$ (see Figure \ref{ChordArcFig}).

Choose $\gamma \delta = \frac{2c}{\epsilon} \delta<\delta_0$. Then,
for $i \geq i'_\delta$, $q_i^- \in B_{\gamma \delta}(q_i^+)$, and
thus the component $\Sigma^\delta_i$ of $B_{\gamma\delta}(q_i^+)\cap
\Sigma_i$ that contains $q_i^+$ is a disk containing $\sigma_i$ and
thus $q_i^-$.
Let $1>\delta_1>0$ be given by the weak chord arc bounds (see
\cite{CY}) and decrease $\delta$, if necessary, so $\gamma
\delta<\frac{1}{2} \delta_0 \delta_1$. Then,
$\mathcal{B}^{\Sigma_i}_{2\gamma\delta \delta_1^{-1}}(q_i^+)$ is a
disk, and so $\Sigma_i^\delta\subset
\mathcal{B}^{\Sigma_i}_{\gamma\delta \delta_1^{-1}}(q_i^+)$, by the
weak chord-arc bounds. Thus, for  $i>i'_\delta$ and $\delta$
sufficiently small, $\dist_{\Sigma_i} (q_i^+,q_i^-)<\gamma
\delta\delta_1^{-1}< \delta_0/3$, which gives a contradiction.

\end{proof}

\begin{figure}
 \centering
 \includegraphics[width=2.5in, bb=0 0 365 360]{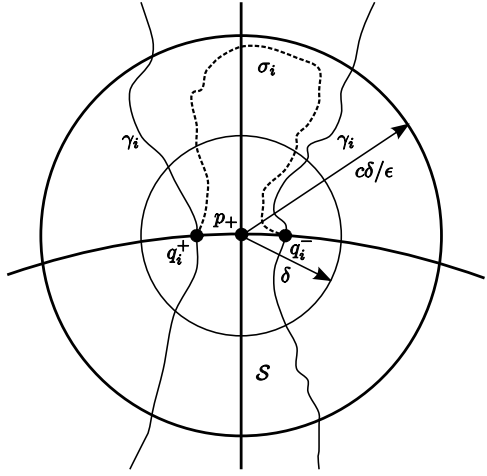}
 \caption{Illustrating the consequence of the one sided curvature estimates}
 \label{ChordArcFig}
\end{figure}

Theorem \ref{WeakCompThmInt} will follow from these uniform
curvature bounds as long as one further point is addressed.  Namely,
we must verify uniform area bounds for the sequence.  For embedded
minimal disks this follows immediately from Lemma 3.1 of \cite{CM2}.
For the more general surfaces we consider, one must use a type of
chord-arc bound, the details of which are in Lemma C.3 of
\cite{CM5}.

\begin{proof} (Theorem \ref{WeakCompThmInt})
By Lemma \ref{boundedcurvlem}, the curvature of the $\Sigma_i$ is
uniformly bounded on any compact subset of $\Real^3$.
Pick $R_j> \Delta$ with $R_j\to \infty$ and let $\Sigma_i^j$ denote
the connected component of $\Sigma_i\cap B_{R_j}(0)$ containing $0$.
Clearly, $\sup_{\Sigma_i^j} |A|^2 \leq C(j)$.  For each $j$, Lemma
C.3 of \cite{CM5} implies the area of $\Sigma_i^j$ is uniformly
bounded.  Hence, by standard compactness theory, for each $j$ a
sub-sequence of the $\Sigma_i^j$ converge smoothly to an embedded,
smooth minimal surface $\Sigma_\infty^j$ with $\partial
\Sigma_\infty^j\subset \partial B_{R_j}$.  Diagonalizing this
sequence, gives a properly embedded surface $\Sigma_\infty$ with
$\Sigma_i^{j_i} \to \Sigma_\infty$.

 Notice $\inj_{\Sigma_i^j} (0)\leq \Delta$ implies $\inj_{\Sigma_\infty}(0)\leq \Delta$, and thus $\Sigma_\infty$ is not a
disk. One may thus argue as in Appendix B of \cite{CM5} to see that
the $\Sigma_i^{j_i}$ converge to $\Sigma_\infty$ with multiplicity
1. Roughly speaking, if the convergence was with a higher degree of
multiplicity, one could construct a positive Jacobi function on
$\Sigma_\infty$ which would force $\Sigma_\infty$ to be stable (by
\cite{FSC}) and hence be a plane by \cite{SchoenStab}. The
convergence can only decrease the genus, so using
Lemma \ref{BndryCompBndProp} one sees $\Sigma_\infty\in
\mathcal{E}(e',g')$ where $0\leq g'\leq g$ and $0\leq e'\leq e+g$.
Finally, doing the same with other components of $B_{R_j}\cap
\Sigma_i$ (if they exist) and noting that the only disjoint
complete, properly embedded minimal surfaces of finite topology are
parallel planes, one has $\Sigma_i \to \Sigma_\infty$.
\end{proof}

\section{Proof of Theorem \ref{CpctnessCorg1}}

\subsection{Scale of the genus} \label{TopDefSec}

We define the extrinsic scale(s)
of the genus:
\begin{defn}\label{outerScaleDef}
For $\Sigma\in \mathcal{E}(1,g,R)$ let
\begin{multline*}
    r_+(\Sigma)=\inf_{x\in B_R}\inf\left\{ r: B_r(x)\subset B_R \right.\mbox{and}  \left. B_r( x) \cap \Sigma \mbox{ has a component of genus $g$} \right\}.
\end{multline*}
We call $r_+(\Sigma)$ the \emph{outer extrinsic scale of the genus}
of $\Sigma$. Furthermore, suppose for all $\epsilon>0$, one of the
components of $B_{r_+(\Sigma)+\epsilon}(x)\cap \Sigma$ has genus
$g$; then we say the genus is \emph{centered at} $x$.
\end{defn}

The outer scale of the genus measures how spread out all the handles
are and the center of the genus should be thought of as a ``center
of mass" of the handles. We also need to measure the scale of
individual handles and where the smallest handle (with respect to
this scale) is located. To that end define:
\begin{defn} \label{InnerScaleLocDef}
For $\Sigma\in \mathcal{E}(1,g,R)$ and $x\in B_{R}(0)$ let
\begin{equation*} r_-(\Sigma,x)=\sup \left\{ r: B_r(x)\subset
B_R(0) \mbox{ and } B_r( x) \cap \Sigma \mbox{ is genus zero} \right\}.
\end{equation*}
If the genus of
$B_r(x)\cap \Sigma$ is zero whenever $B_r(x)\subset B_R(0)$,  set
$r_-(\Sigma,x)=\infty$.
\end{defn}
\begin{defn}\label{innerScaleDef}
For $\Sigma\in \mathcal{E}(1,g,R)$ let
\begin{equation*}
    r_-(\Sigma)=\inf_{x\in B_R(0)} r_-(\Sigma,x).
\end{equation*}
We call $r_-(\Sigma)$ the \emph{inner extrinsic scale of the genus}
of $\Sigma$.
\end{defn}
\begin{rem}
One easily checks that $R>r_+(\Sigma)\geq r_-(\Sigma)$ for $\Sigma\in
\mathcal{E}(1,g,R)$, with equality holding if $g=1$. When $g=1$ we denote the common value by $r(\Sigma)$.
\end{rem}

\subsection{Two-sided bounds on the genus}\label{twosidedsec}
Using the no-mixing theorem of \cite{CM5}, we show that uniform
control on both the inner and outer extrinsic scales of the genus
implies a uniform lower bound on the injectivity radius. By the weak
chord-arc bounds of \cite{CY}, uniform extrinsic control also gives
a point with a uniform upper bound on its injectivity radius. We
then apply Theorem \ref{WeakCompThmInt} to obtain smooth
sub-sequential convergence to some complete embedded minimal surface
in $\Real^3$. The uniform bound on the outer scale of the genus
implies this limit must be a genus $g$-helicoid.



\begin{lem}\label{injectivelb} Fix $0<\alpha\leq 1$ and
$g\in \mathbb{Z}^+$. Then there exists $R_0>1$ and $1>\delta_0>0$
depending only on $\alpha$ and $g$ so: If $\Sigma\in
\mathcal{E}(1,g,R)$ with $R\geq R_0$, $1= r_-(\Sigma)\geq \alpha
r_+(\Sigma)$, and the genus of $\Sigma$ is centered at $0$, then as
long as $B_{\delta_0}(x)\subset B_{R}$, every component of
$B_{\delta_0} (x)\cap \Sigma$ is a disk.
\end{lem}
\begin{proof}
Suppose the lemma was not true. Then, there exists a sequence of
surfaces $\Sigma_i \in \mathcal{E}(1,g,R_i)$ with
$1=r_-(\Sigma_i)\geq \alpha r_+(\Sigma_i)$, $R_i\to \infty$, and the
genus of $\Sigma_i$ centered at $0$. Further, there exist points
$x_i$ and a sequence $\delta_i \to 0$ so that one of the components
of $B_{\delta_i}(x_i)\cap \Sigma_i$ is not a disk.
Notice for fixed $x$ and $r$, if $\bar{B}_r(x)\cap
\bar{B}_{\alpha^{-1}}(0)=\emptyset$ with $i$ large enough so
$B_r(x)\subset B_{R_i}$, then each component of $B_r(x)\cap
\Sigma_i$ is a disk. Thus, we may assume $x_i\in
B_{3/2\alpha^{-1}}$. Because uniform curvature bounds would imply
a uniform extrinsic scale on which $\Sigma_i$ is graphical, we see
that $\sup_{B_{2\alpha^{-1}} \cap \Sigma_i} |A|^2 \to \infty$.
Hence, by Theorem 0.14 of \cite{CM5}, up to to a sub-sequence, the
$\Sigma_i$ convergence to a singular lamination $\mathcal{L}$.

Let us now determine $\mathcal{L}$ and see that this gives a
contradiction.  By possibly passing to a further sub-sequence, we may
assume that $x_i\to x_\infty$.  Pick $i_0$ large enough so that
$|x_i-x_\infty|+\delta_i\leq 1/8$ for all $i\geq i_0$. Then for
$i\geq i_0$, $B_{\delta_i}(x_i)\subset B_{1/2} (x_\infty)$ and thus
$B_{1/2} (x_\infty)\cap \Sigma_i$ contains a non-disk
component. As $r_-(\Sigma_i)>1/2$, this component has genus zero. Hence, the boundary of this component is not
connected, and so $x_\infty$ is a point of
$\mathcal{S}_{neck}$ of the lamination $\mathcal{L}$ (see Definition \ref{NeckDef}).

For any ball $B_r(x)$ with $\bar{B}_r(x)\cap
\bar{B}_{\alpha^{-1}}(0)=\emptyset$ and $i$ large enough so
$B_r(x)\subset B_{R_i/2}(0)$, one has that all components of
$B_r(x)\cap \Sigma_i$ are disks, and thus the maximum principle and
the no-mixing theorem of \cite{CM5} imply that the singular set
$\mathcal{S}$ of $\mathcal{L}$ is contained in $B_{\alpha^{-1}}(0)$.
As a consequence, one may rotate so that $\mathcal{L} \subset
\set{|x_3|\leq \alpha^{-1}}$. Thus, for any $k\in \mathbb{N}$ there
is an $i_k$ so that for $i\geq i_k$, $\Sigma_{i}\cap B_{k}\subset \{|x_3| \leq
2\alpha^{-1}\}$ and $R_{i}> k^2$.  Now set
$\tilde{\Sigma}_k=\frac{1}{k} \Sigma_{i_k}$, so $\tilde{\Sigma}_k$
is a new sequence with the genus still centered at $0$, $\partial
\tilde{\Sigma}_k \cap B_{k}(0)=\emptyset$,
$r_+(\tilde{\Sigma}_k)\leq \alpha^{-1}/k\to 0$, and
$\tilde{\Sigma}_{l}\cap B_1\subset \{|x_3| \leq 2(\alpha k)^{-1}\}$ for $l\geq
k$. Clearly the curvature in $B_{2\alpha^{-1}}$ is still blowing
up and so by possible passing to a sub-sequence we have convergence
to a singular lamination $\tilde{\mathcal{L}}$. But
$r_{+}(\tilde{\Sigma}_k)\to 0$ implies that $0\in
\mathcal{S}_{ulsc}$. However, $\tilde{\Sigma}_l \cap B_1$ converges (in a Hausdorff sense) to $\set{x_3=0}\cap B_1$, contradicting the convergence of Theorem 0.9 of
\cite{CM5}.
\end{proof}

We now bound the injectivity radius of $\Sigma$ above in terms of $r_+(\Sigma)$.

\begin{lem} \label{InjUBLem}
There exists $R_0,\Delta$ with $R_0\geq 5\Delta>10>0$, so: If
$\Sigma\in \mathcal{E}(1,g,R)$, where $R>R_0$, and one of the
components, $\Sigma'$, of $\Sigma\cap B_{2}(0)$ has positive genus,
then the injectivity radius at all points of $\Sigma'$ is bounded
above by $\Delta$.
\end{lem}
\begin{proof}
Pick $0< \delta_1 \leq 1/2$ as in the weak chord-arc bounds of
\cite{CY}. We claim that we may choose $\Delta={4/\delta_1}\geq 3$
and $R_0=5 \Delta$.  To see this, suppose that $\Sigma$ satisfies
the hypotheses of the lemma for some $R>R_0$, and $\Sigma'$ is a
component of $\Sigma\cap B_{2}$ with positive genus, but the
injectivity radius of some point $x\in \Sigma'$ is (strictly)
bounded below by $\Delta$. Then $\mathcal{B}^\Sigma_{\Delta}(x)$ is
disjoint from the boundary of $\Sigma$ and is topologically a disk.
Thus, the weak chord-arc bounds of \cite{CY} imply that the
component of $B_4(x)\cap \Sigma$ containing $x$, $\Sigma_{x,4}$, is
contained in the disk $\mathcal{B}^\Sigma_{\Delta/2}(x)$. The
maximum principle implies $\Sigma_{x, 4}$ is itself a disk; but
$\Sigma'\subset \Sigma_{x,4}$, which gives a contradiction.
\end{proof}
\begin{cor}\label{InjectiveubCor}
Fix $1\geq\alpha>0$ and let $\Sigma\in \mathcal{E}(1,g,R)$ and
suppose that $1=r_-(\Sigma)\geq \alpha r_+(\Sigma)$, $R\geq
R_0\alpha^{-1}$, the genus is centered at $0$, and $R_0,\Delta$
are as above.  Then there is a point $p_0\in \Sigma\cap
B_{\alpha^{-1}} $ with $\inj_{\Sigma} (p_0)\leq
\Delta\alpha^{-1}$.
\end{cor}
We can now prove our desired compactness result.
\begin{proof} (Theorem \ref{CpctnessCorg1})
By Lemma \ref{injectivelb}, the injectivity radius of the sequence
$\Sigma_i$ is uniformly bounded below by $\delta_0>0$. Moreover, by
Corollary \ref{InjectiveubCor} there is a point $p_i$ in the ball
$B_{\alpha^{-1}}(0)$ so that $\inj_{\Sigma_i} (p_i)\leq \Delta
\alpha^{-1}$. As a consequence, we may apply Theorem
\ref{WeakCompThmInt} and obtain  a sub-sequence of the $\Sigma_i$
that converges, with multiplicity one, uniformly in $C^\infty$ on compact subsets of
$\Real^3$ to some complete, embedded
non-simply connected minimal surface $\Sigma_\infty$.  The inner and outer control
on the genus and topology of the $\Sigma_i$ carry over to
the limit in the indicated manner.
\end{proof}

\section{Applications}
The compactness result developed in the previous section is
particularly strong for sequences of genus-one surfaces, as there is
only one scale for the genus. Indeed, there are several interesting
scales for genus-$g$ helicoids, the scale of the asymptotic helicoid
and the scale(s) of the genus. In principle, one might expect a
relationship between them. This is the case for genus one; however it is not clear
there is such a connection for $g\geq 2$ -- this is discussed in more detail in \cite{BBCompLam}.
\subsection{Compactness of $\mathcal{E}(1,1)$}\label{GlobalSec}
Let us now focus on the space $\mathcal{E}(1,1)$, i.e. genus-one
helicoids.  To prove Theorem \ref{MainCpctnessThm}, we show that for
any $\Sigma\in \mathcal{E}(1,1)$, there is an upper and lower bound
on the ratio between the scale of the genus and the scale of the
asymptotic helicoid.
To verify this bound and that the limit is also asymptotic to $H$,
we must make a connection between the local nature of the
convergence and global nature of the asymptotic scale. This is made
by evaluating certain path integrals of holomorphic Weierstrass
data. Indeed, we first establish a uniform $R$ such that all
vertical normals of each $\Sigma_i$ (after a rotation) occur in
$B_R(0)$.  We then find annular ends $\Gamma_i$, conformally mapped
to the same domain in $\mathbb{C}$ by $z_i=(x_3)_i+\sqrt{-1}
(x_3^*)_i$ and with Weierstrass data as described in Corollary 1.2
of \cite{BB2}. Finally, we use calculus of residues on $\partial
\Gamma_i$ to establish uniform control on the center and radius of
the genus for a sub-sequence.

Note that Theorem 1.1 of \cite{BB2} implies that for $\Sigma \in
\mathcal{E}(1,1)$, there are two points where the Gauss map points
parallel to the axis of the asymptotic helicoid. Thus, after
translating, rescaling and rotating, any $\Sigma\in
\mathcal{E}(1,1)$ will satisfy the conditions of the following
lemma.
\begin{lem} \label{PosZeroPoleLem}
Suppose $\Sigma\in \mathcal{E}(1,1)$, the genus of $\Sigma$ is
centered at $0$, $r(\Sigma)=1$, and that $g$, the usual
stereographic projection of the Gauss map of $\Sigma$, has only a
single pole and single zero. Then there is an $R\geq 1$ independent
of $\Sigma$ so that the pole and zero of $g$ lie in $B_R\cap
\Sigma$.
\end{lem}

\begin{proof}
Suppose this was not the case. That is, one has a sequence
$\Sigma_i\in \mathcal{E}(1,1)$ so that the genus of each $\Sigma_i$
is centered at $0$, $r(\Sigma_i)=1$, and each $g_i$ has a single
pole and single zero, at least one of which does not lie in
$B_i$. By rotating $\Sigma_i$, we may assume the pole does not
lie in $B_i$.  By Theorem \ref{CpctnessCorg1}, a sub-sequence of
the $\Sigma_i$ converge uniformly in $C^\infty$ on compact subsets
of $\Real^3$, with multiplicity 1, to $\Sigma_\infty\in
\mathcal{E}(1,1)$ where $r(\Sigma_\infty)=1$.

Now consider
$g_\infty:\Sigma_\infty \to \mathbb{C}$; it has at least one zero.
Suppose $g_\infty$ has more than one zero and call two such zeros
$q_1,q_2 \in \Sigma_\infty$. For $\delta_0$ as in Lemma
\ref{injectivelb}, denote by $\sigma_\infty^j$ a closed, embedded
curve in the component of $B_{\delta_0/2}(q_j)\cap \Sigma_\infty$
that contains $q_j$, chosen so that $\sigma_\infty^j$ surrounds
$q_j$ but neither surrounds nor contains any other pole or zero of
$g_\infty$. Additionally, we choose the two $\sigma_\infty^j$ to be
disjoint. Thus, $\int_{\sigma_\infty^j}
\frac{dg_\infty}{g_\infty}=2\pi \sqrt{-1}$.

By the convergence, there are simple, closed curves
$\sigma_i^j$ ($j=1,2$) in $\Sigma_i$ so that $\sigma_i^j$ converges
smoothly to $\sigma_\infty^j$.  Let $D_i^j \subset \Sigma_i$ denote
a disk such that $\partial D_i^j = \sigma_i^j$. For large enough
$i$, $D_i^1 \cap D_i^2 = \emptyset$. There is at most one zero of
$g_i$ in $\cup_{j=1,2}D_i^j$ and the pole of $g_i$ lies outside
$B_i$. Thus, there exists $i'$ such that, for $i\geq i'$, either
$\int_{\sigma_i^1} \frac{dg_i}{g_i}=0$ or the same is true for
$\sigma_i^2$. Letting $i \to \infty$ gives a contradiction; so
$g_\infty$ has only one zero. By Proposition 4.2 of \cite{BB2},
$g_\infty$ has only one pole.

To conclude the proof, note there is an $R_1$ such that the
pole, $p_\infty$, of $g_\infty$ is in $B_{R_1} \cap \Sigma_\infty$.
Let $\sigma_\infty\subset \Sigma_\infty$ be a smooth, embedded, closed curve
in $B_{\delta_0/2}(p_\infty)$ that surrounds $p_\infty$ and neither
contains nor surrounds the zero of $g_\infty$. Thus,
$\int_{\sigma_\infty}\frac{dg_\infty}{g_\infty}=-2\pi\sqrt{-1}$.
Then by our convergence result, there are smooth, embedded, closed curves
$\sigma_i \subset \Sigma_i \cap B_{3\delta_0/4} (p_\infty)$, with
$\sigma_i$ converging to $\sigma_\infty$. Note $\sigma_i$ is
necessarily null-homotopic and
for $i$ large enough so $\sigma_\infty \subset B_{i/2} (0)$ we
compute $\frac{1}{2\pi \sqrt{-1}} \int_{\sigma_i}
\frac{dg_i}{g_i}\geq 0$.  As $i\to \infty$ the smooth convergence
provides a contradiction.
\end{proof}
Before proving, Theorem
\ref{MainCpctnessThm}, we establish the necessary uniform
control on the center and radius of the genus for surfaces in
$\mathcal{E}(1,1)$ that are asymptotic to a fixed helicoid $H$.
Throughout the following proof, we make repeated use of
\cite{BB2}.
\begin{lem}
Let $\Sigma_i \in \mathcal{E}(1,1)$ and suppose that all the
$\Sigma_i$ are asymptotic to the same helicoid, $H$, which has axis the $x_3$-axis.  Then, there
exist $C_1,C_2>0$ and a sub-sequence such
that $1/C_1 \leq r(\Sigma_{i})\leq C_1$ and, after a rotation,
$|x_1(p_{i})|+|x_2(p_i)| \leq C_2$, where $p_{i}$ is the center of
the genus of each $\Sigma_{i}$.
\end{lem}

\begin{proof}
Translate each $\Sigma_i$ by $-{p}_i$ so that the genus of
each of the $\Sigma_i$ is centered at 0.  Then rescale each
$\Sigma_i$ by $\alpha_i$ so that after the rescaling
$r(\Sigma_i)=1$. Thus, each rescaled and translated surface
$\Sigma_i$ is asymptotic to the helicoid $H_i=\alpha_i
(H-{p}_i)$.
By Theorem \ref{CpctnessCorg1}, passing to a sub-sequence,
$\Sigma_i$ converges uniformly on compact sets to $\Sigma_\infty\in
\mathcal{E}(1,1)$ with multiplicity one.  By Corollary 1.3 of
\cite{BB2}, $\Sigma_\infty$ is asymptotic to some helicoid
$H_\infty$. Since each $\Sigma_i$ is asymptotic to $H_i$, which has
vertical axis, the stereographic projection $g_i$ of the Gauss map
of $\Sigma_i$ has exactly one pole and one zero. Thus, as in the
proof of Lemma \ref{PosZeroPoleLem}, we can conclude the same for
$\Sigma_\infty$.

As each $\Sigma_i$ satisfies the hypotheses of Lemma
\ref{PosZeroPoleLem}, there exists an $R$ such that both the zero
and pole of $g_i$ ($i \leq \infty$) lie on the component,
$\Sigma_i^0$, of $B_{R}(0)\cap \Sigma_i$ containing the genus. Thus,
 $\Gamma_i=\Sigma_i\backslash \bar{\Sigma}_i^0$ is an annulus and $g_i$ has no poles or zeros in
$\Gamma_i$.  Hence, by the arguments of Section 4 of \cite{BB2},
$z_i=(x_3)_i+\sqrt{-1}(x_3)_i^*:\Gamma_i \to \mathbb{C}$ and
$f_i=\log g_i : \Gamma_i \to \mathbb{C}$ are well-defined. (Here
$(x_3)_i$ is $x_3$ restricted to $\Sigma_i$ and $(x_3)_i^*$ is the
harmonic conjugate of $(x_3)_i$.) Corollary 1.2 of \cite{BB2} tells
us that $f_i(p)=\sqrt{-1} \lambda_i z_i(p)+F_i(p)$ where $\lambda_i$
determines the scale of $H_i$ and $F_i(p)$ is
holomorphic on $\Gamma_i$ and extends holomorphically to
$\infty$ with $F(\infty)=0$.

Let $A_C= \set{z\in \mathbb{C}: |z|>C}$. By suitably translating
$(x_3)_i^*$, $A_{C/2} \subset z_\infty(\Gamma_\infty)$ for $C>0$;
thus there exists $i_0$ such that, for $i_0\leq i \leq \infty$,
 $A_C \subset z_i(\Gamma_i).$
By precomposing with $z_i^{-1}$, we may think of $f_i$ as a
holomorphic function on $A_C$. If $u$ is the standard coordinate of
$\mathbb{C}$ restricted to $A_C$, then for $i_0\leq i \leq \infty$,
$f_i(u)=\sqrt{-1} \lambda_i u+F_i(u)$ where $F_i$ is holomorphic and extends holomorphically and with a zero to
$\infty$.
There is an $R'> R$ so that
$\set{p\in\Gamma_\infty:|z_\infty(p)|\leq 2C}\subset B_{R'}$;
thus, by perhaps increasing $i_0$, for $i_0\leq i \leq \infty$,
$\set{p\in\Gamma_i:|z_i(p)|\leq 2C}\subset B_{2R'}$. As a
consequence, the uniform convergence of $\Sigma_i$ to
$\Sigma_\infty$ implies that for $\gamma=\set{u:|u|=\frac{3}{2}
C}\subset A_C$, $f_i\to f_\infty$ in $C^\infty(\gamma)$.  On the
other hand, the calculus of residues implies that for $i_0\leq i
\leq \infty$,
\begin{equation} \label{ScaleIntegral}
 \int_\gamma f_i(u) \frac{du}{u^2} =-2\pi \lambda_i
\end{equation}
and hence we see immediately that $\lambda_i \to \lambda_\infty>0$.
Since the initial helicoid $H$ had some $\lambda_H$ associated with
its Weierstrass data, this gives an upper and lower bound on the
rescaling of each initial $\Sigma_i$, thus producing the necessary
$C_1$.

Since each $F_i$ is holomorphic on $A_C$ with a holomorphic
extension to $\infty$ (and a zero there), we can expand in a Laurent
series, i.e. for every $u\in A_C$ one has
\begin{equation}
 F_i(u)=\sum_{j=1}^\infty \frac{a_{i,j}}{u^j}
\end{equation}
where this is a convergent sum.
Thus, for $i_0\leq i\leq \infty$
\begin{equation}
 \int_\gamma f_i(u) u^{j-1} du =2\pi\sqrt{-1} a_{i,j}
\end{equation}
and hence $\lim_{i\to \infty} a_{i,j}=a_{\infty,j}$. Thus, $F_i\to
F_\infty$ uniformly in $\set{u: 2C\leq |u|\leq \infty}$.

The Weierstrass representation implies that $H_i$ converge to
$H_\infty$ and hence that $x_1({p}_i)\to x_1({p}_\infty)$ and
$x_2({p}_i)\to x_2({p}_\infty)$.  This produces the necessary bound,
$C_2$, on the distance between the center of the genus and the axis
of the initial helicoid $H$. Indeed, the convergence $F_i \to
F_\infty$ implies there is a uniform $C_0$ so that for $u>2C$ and
$i_0\leq i\leq \infty$, one can write
$F_i(u)=\frac{a_{i,1}}{u}+\tilde{F}_i(u)$, where $|\tilde{F}_i
(u)|\leq \frac{C_0}{u^2}$ and $|a_{i,1}|\leq C_0$.  Recall the
Weierstrass representation gives that
\begin{equation}
dx_1-\sqrt{-1} dx_2 =g^{-1} dh-\bar{g}\bar{dh},
\end{equation}
where $dh=dx_3+\sqrt{-1}dx_3^*$ is the height differential.
Integrating this form along $\set{t+(\sqrt{-1})0: t\in
[3C,t_1]}\subset A_C$, gives:
\begin{equation}
((x_1)_i(t_1)-(x_1)_i(3C))-\sqrt{-1}
((x_2)_i(t_1)-(x_2)_i(3C))=\int_{3C}^{t_1}  J_i(t) dt
\end{equation}
where $J_i(t)=e^{-\sqrt{-1}(\lambda_i t+\im F_i(t))}(e^{-\re
F_i(t)}-e^{\re F_i(t)}) $. For $C$ sufficiently large, $|F_i(t)|\leq
1/2$ and so expanding in a power series,
\begin{equation}
J_i(t)=-2 e^{-\sqrt{-1}\lambda_i t} \frac{\re a_{i,1}}{t}+G_i(t)
\end{equation}
where here $|G_i(t)|\leq \frac{10 C_0}{t^2}$.

The first term is a convergent (as $t_1\to \infty$) oscillating integral while the integral of $G_i(t)$ is absolutely convergent.  Thus, there exists
a $C_2$ depending only on $C_0$ (and in particular independent of
$t_1, i$) so, for $i_0\leq i\leq \infty$, on each $\Sigma_i$,
\begin{equation}
 |(x_1)_i(t_1)-(x_1)_i(3C)|+|(x_2)_i(t_1)-(x_2)_i(3C)|\leq C_2.
\end{equation}
Recall the translated surfaces $\Sigma_i$ have genus centered at
zero. Since one finds the axis by letting $t_1 \to \infty$, the
above bound shows each translated surface has axis a uniform
distance from the origin. Thus, the original center for each genus,
$p_i$, satisfies the desired uniform bound.
\end{proof}

The previous lemma tells us that for any sequence of surfaces in
$\mathcal{E}(1,1)$, asymptotic to a fixed helicoid, there is a
sub-sequence with the scale of the genus uniformly controlled and
the center of the genus lying inside of a cylinder.  With this
information, we can now show Theorem \ref{MainCpctnessThm}:
\begin{proof}
By the previous lemma, we know that there exists an $R>0$ and a
sub-sequence $\Sigma_i$ such that $r(\Sigma_i) \to r_\infty >0$ and
the genus of these $\Sigma_i$ is centered in a cylinder of radius
$R$. Translate the surface parallel to the axis of the cylinder by
some $v_i$ so the genus is centered in $B_{R}(0)$. Use Theorem
\ref{CpctnessCorg1} to show that there is a sub-sequence so
$\Sigma_i-v_i \to \Sigma_\infty \in \mathcal{E}(1,1)$, which is
asymptotic to some helicoid.  The integral \eqref{ScaleIntegral}
implies this helicoid is $H$.  If $\lim_{i\to \infty} |v_i|<\infty$
then $\Sigma_\infty \in \mathcal{E}(1,1)$, and if $\lim_{i\to
\infty}|v_i|=\infty$ then $\Sigma_\infty = H$.
\end{proof}

\subsection{Geometric Structure of $\mathcal{E}(1,1,R)$} \label{consequencesection}
Theorem 1.1 of \cite{BB1} gave an effective description of disks
with large curvature by comparing them with helicoids. Theorem
\ref{CpctnessCorg1} allows us to do the same for elements
$\mathcal{E}(1,1,R)$.
\begin{thm}\label{effresult}
Given $\epsilon>0$ and $R\geq 1$ there exists an
$R'=R'(\epsilon,R)\geq R$ so that: if $\Sigma\in
\mathcal{E}(1,1,R')$ with $r(\Sigma)=1$ and the genus of $\Sigma$ is
centered at $0$, then the component of $B_{R}(0)\cap \Sigma$
containing the genus is bi-Lipschitz with a subset of an element of
$\mathcal{E}(1,1)$ and the Lipschitz constant is in
$(1-\epsilon,1+\epsilon)$.
\end{thm}
\begin{proof}
The proof follows by contradiction exactly as in the proof of
Theorem 1.1 of \cite{BB1} with two very minor differences.  First,
Theorem 1.3 of this paper replaces the convergence to
$\Sigma_\infty$ outlined in that proof. Second, by Theorem 1.1 of
\cite{BB2}, $C=\sup_{\Sigma_{\infty}}|A|<\infty$ and so for any
$\epsilon$, we find a smooth $\nu_i$ defined on a subset of
$\Sigma_\infty$ so that $C|\nu_i| + |\nabla_{\Sigma_\infty}
\nu_i|<\epsilon$ and the graph of $\nu_i$ is the component of
$\Sigma_i \cap B_{R'}$ containing the genus.  For details, see
Section 6 of \cite{BB1}.
\end{proof}
\appendix
\section{The lamination theory of \cite{CM5}} \label{CM5Sec}
In \cite{CM5}, Colding and Minicozzi generalize their lamination
theory for minimal disks -- i.e. Theorem 0.1 of \cite{CM4} -- to
arbitrary sequences of minimal surfaces of fixed, finite genus.  To
do so, they must allow for a more general class of singular models,
as is clear from considering a rescaling of the catenoid.  Surfaces
modeled on the neck of a catenoid form an important class, one that
is characterized by having genus zero and disconnected boundary.
Another important class, especially for our purposes, are surfaces
of finite genus and connected boundary.  This second class of
surfaces, because they have connected boundary, turn out to be
structurally very similar to disks. Indeed, most of the results of
\cite{CM1,CM2,CM3,CM4} hold for them (in a suitably form) and with
only slight modifications of the proofs; this has important
ramifications for lamination theory.

We first briefly discuss the most general lamination result of
\cite{CM5} -- Theorem 0.14.  Suppose $\Sigma_i$ is a sequence of
embedded minimal surfaces with fixed, finite genus and $\partial
\Sigma_i \subset B_{R_i}$ with $R_i\to \infty$.  Colding and
Minicozzi show that if the curvature of the sequence blows up at a
point $y \in \Real^3$ (i.e. if for all $r>0$, $\sup_i
\sup_{B_r(y)\cap \Sigma_i} |A|^2=\infty$), then after a rotation, a
sub-sequence $\Sigma_i$ converge to the singular lamination
$\mathcal{L}\backslash \mathcal{S}$ in the $C^\alpha$ topology
($\alpha \in (0,1)$) and the curvature blows precisely at points of
$\mathcal{S}$.  Here $\mathcal{L}=\{x_3=t\}_{t \in \mathcal{I}}$,
$\set{x_3(y):y\in \mathcal{S}}=\mathcal{I}$ a closed subset of
$\Real^3$.  Notice little global information about the limit
lamination is given; this contrasts with the case for disks where
$\mathcal{I}=\Real$ and $\mathcal{S}$ is a Lipschitz graph over the
$x_3$-axis -- in fact a line.

More generally, the topology of the sequence can restrict
$\mathcal{I}$ and give more information about convergence near
$\mathcal{S}$ and the structure of $\mathcal{S}$. To make this
precise, we distinguish between two types of singular points $
y\in\mathcal{S}$. Heuristically, the distinction is between points
where
on
small scales near the point all the $\Sigma_i$ are disks and points
where
on small scales near the point all the $\Sigma_i$ contain necks.
This is the exact description if the genus of the surfaces is zero,
but must be
 refined for sequences with positive genus.  Following \cite{CM5}, we
make this precise for a sequence $\Sigma_i$  converging to a
lamination $\mathcal{L}$ with singular set $\mathcal{S}$:
\begin{defn}\label{ULSCdef}
 We say $y\in \mathcal{S}$ is an element of $\mathcal{S}_{ulsc}$ if there
exist both $r>0$ fixed and a sequence $r_i \to 0$ such that $B_r(y)
\cap \Sigma_i$ and $B_{r_i}(y) \cap \Sigma_i$ have the same genus
and every component of $B_{r_i}(y)\cap \Sigma_i$ has connected
boundary.
\end{defn}
\begin{defn}\label{NeckDef}
 We say $y\in \mathcal{S}$ is an element of $\mathcal{S}_{neck}$ if there
exist both $r>0$ fixed and a sequence $r_i \to 0$ such that $B_r(y)
\cap \Sigma_i$ and $B_{r_i}(y) \cap \Sigma_i$ have the same genus
and $B_{r_i}(y)\cap
\Sigma_i$ has at least one component with disconnected boundary.
\end{defn}
For instance, if the $\Sigma_i$ are the homothetic blow-down of a
helicoid or genus-one helicoid, then $0$ is an element of
$\mathcal{S}_{ulsc}$, whereas if the $\Sigma_i$ are the homothetic
blow-down of a catenoid then $0$ is an element of
$\mathcal{S}_{neck}$. Colding and Minicozzi show that near a point
of $\mathcal{S}_{ulsc}$ or $\mathcal{S}_{neck}$ these examples model
the behavior of the convergence.

A major result of \cite{CM5} refines the general compactness theorem
based on more careful analysis of the topology of the sequence. One
important example is the no-mixing theorem (i.e. Theorem 0.4 of
\cite{CM5}), which states that, up to a passing to a sub-sequence,
either $\mathcal{S}=\mathcal{S}_{ulsc}$ or
$\mathcal{S}=\mathcal{S}_{neck}$. This is particularly important as
ULSC sequences (i.e. sequences where
$\mathcal{S}=\mathcal{S}_{ulsc}$) behave much like disks. Indeed,
Theorem 0.9 of \cite{CM5} states:
\begin{thm} \label{CM509Thm}
Let $\Sigma_i \subset B_{R_i} = B_{R_i}(0) \subset \Real^3$ be a
sequence of compact embedded minimal surfaces of fixed genus with
$\partial \Sigma_i \subset \partial B_{R_i}$ where $R_i \to \infty$.
 Assume the sequence $\Sigma_i$ is ULSC. If $\sup_{B_1 \cap \Sigma_i} |A|^2 \to \infty$, then there exists
a sub-sequence $\Sigma_j$, a foliation $\mathcal{L}=\{x_3=t\}_t$ of
$\Real^3$ by parallel planes, and a closed nonempty set
$\mathcal{S}$ in the union of the leaves of $\mathcal{L}$ such that
after a rotation of $\Real^3$:
\begin{enumerate}
\item For each $\alpha\in (0,1)$, $\Sigma_j \backslash \mathcal{S}$
converges in the $C^\alpha$-topology to
$\mathcal{L}\backslash \mathcal{S}$.
\item $\sup_{B_r(x)\cap \Sigma_j}|A|^2 \to \infty$ as $j \to \infty$
for all $r>0$ and $x \in \mathcal{S}$.

\item $\mathcal{S}$ is orthogonal to the
leaves of the lamination and is
either the union of two disjoint lines $\mathcal{S}_1$ and
$\mathcal{S}_2$, or is a single line.
\item \label{twosing}Away from $\mathcal{S}$, each $\Sigma_j$
consists of exactly two multi-valued graphs spiraling together. Near $\mathcal{S}$ the pair of multi-valued graphs
form double spiral stair-cases with opposite orientations.  Thus, circling only
one component of $\mathcal{S}$ results in going either up
or down.  If $\mathcal{S}=\mathcal{S}_1\cup \mathcal{S}_2$ then the multi-valued graphs are glued so that a path circling both $\mathcal{S}_1$ and
$\mathcal{S}_2$ closes up.
\end{enumerate}
\end{thm}


\begin{rem}\label{twosingremark}The situation with two singular sets is modeled on the degeneration of the Riemann examples.  As such the limit can be thought of as having ``infinite'' topology.  This is made rigorous by noting that for $R>0$ large enough so $B_{R/2}(0)$ intersects both singular
lines, then \eqref{twosing} implies that for any $f\in \mathbb{Z}^+$ for $i$ sufficiently large $B_R(0)\cap \Sigma_i$ has a component with at least $f$ boundary components.
\end{rem}

\bibliographystyle{amsplain}
\bibliography{thesisbib}

\end{document}